\documentclass[11pt]{amsart}
\usepackage{amsxtra, amsfonts, amsmath, amsthm, amstext, amssymb, amscd, mathrsfs, verbatim, color}
\usepackage{threeparttable}
\usepackage{mathtools}
\usepackage[ansinew]{inputenc}\usepackage[T1]{fontenc}
\usepackage[all,cmtip]{xy}
\theoremstyle{plain}
\newtheorem{theorem}{Theorem}[section]
\newtheorem{lemma}[theorem]{Lemma}
\newtheorem{definition-theorem}[theorem]{Definition-Theorem}
\newtheorem{proposition}[theorem]{Proposition}
\newtheorem{corollary}[theorem]{Corollary}

\theoremstyle{definition}
\newtheorem{definition}[theorem]{Definition}
\newtheorem{example}[theorem]{Example}
\newtheorem{remark}[theorem]{Remark}
\newtheorem{notation}[theorem]{Notation}
\newcommand \bth[1] { \begin{theorem}\label{t#1} }
\newcommand \ble[1] { \begin{lemma}\label{l#1} }

\newcommand \bpr[1] { \begin{proposition}\label{p#1} }
\newcommand \bco[1] { \begin{corollary}\label{c#1} }
\newcommand \bde[1] { \begin{definition}\label{d#1}\rm }
\newcommand \bex[1] { \begin{example}\label{e#1}\rm }
\newcommand \bre[1] { \begin{remark}\label{r#1}\rm }

\newcommand \bnota[1] {\begin{notation}\label{n#1}\rm }
\newcommand {\ele} { \end{lemma} }

\newcommand {\epr} { \end{proposition} }
\newcommand {\eco} { \end{corollary} }
\newcommand {\ede} { \end{definition} }
\newcommand {\eex} { \end{example} }
\newcommand {\ere} { \end{remark} }
\newcommand {\enota} { \end{notation} }

\newcommand {\Alg} {\mathrm{Alg}}
\newcommand {\GL} {\mathrm{GL}}
\newcommand {\Wh} {\mathrm{Wh}}
\newcommand {\St} {\mathrm{St}}
\newcommand {\Ext} {\mathrm{Ext}}
\newcommand {\EP} {\mathrm{EP}}
\newcommand {\Hom} {\mathrm{Hom}}
\newcommand {\ind} {\mathrm{ind}}

\title[Vanishing Ext-groups for $(\GL_{n+1}(F), \GL_n(F))$]{A vanishing Ext-branching theorem for $(\GL_{n+1}(F), \GL_n(F))$}

\author[Chan]{Kei Yuen Chan}

\address{
Shanghai Center for Mathematical Sciences \\
Fudan University}
\email{kychan@fudan.edu.cn}

\author[Savin]{Gordan Savin}
\address{
Department of Mathematics \\
University of Utah}
\email{savin@math.utah.edu}

\begin{document}
\begin{abstract} 
We prove a conjecture of Dipendra Prasad on the Ext-branching problem from $\GL_{n+1}(F)$ to $\GL_n(F)$, where $F$ is a $p$-adic field. 
\end{abstract} 

\maketitle

\section{Introduction}

Decomposing a smooth representation of $\GL_{n+1}(F)$, when restricted to $\GL_n(F)$, is a well known and well studied problem that was initiated in a paper of Prasad \cite{Pr1}. 
 Today, this problem is a part of a large family of Gan-Gross-Prasad restriction problems \cite{GGP} that are at the center of much research in Representation Theory and Number Theory. 
In order to describe what is known, and what is new in this paper, 
 let  $G_n=\GL_n(F)$ and let $\Alg(G_n)$ be the category of smooth representations of $G_n$. 
For every $\pi\in \Alg(G_n)$, let 
$\Wh(\pi)$ be the space of Whittaker functionals on $\pi$.  If $\pi$ is irreducible then $\Wh(\pi)$ is one or zero dimensional. We say that $\pi$ is generic or degenerate, 
respectively.  Let $\pi_1$ be an irreducible representation of $G_{n+1}$.  One of the most significant results in the subject is that 
the restriction of $\pi_1$ to $G_n$ is multiplicity free \cite{AGRS}, \cite{AG}, \cite{SZ}, that is, for every irreducible representation $\pi_2$ of $G_n$, 
\[ \dim \Hom_{G_n}(\pi_1, \pi_2) \leq 1\] 
 and it is one if both representations are generic. On the other hand, Dipendra Prasad has proved 
in \cite{Pr} the following beautiful formula: 
\[ 
\EP(\pi_1, \pi_2) := \sum (-1)^i \dim\Ext^i_{G_n}(\pi_1,\pi_2) = \dim \Wh(\pi_1) \cdot \dim \Wh(\pi_2). 
\] 
In particular, the formula implies that  $\EP(\pi_1, \pi_2)=1$ if both representations are generic.  Since $\dim \Hom_{G_n}(\pi_1, \pi_2)=1$, Prasad has conjectured that 
$\Ext^i_{G_n}(\pi_1,\pi_2)$ vanish for $i>0$ if both representations are generic.  

\smallskip 

The first main result in this paper is a proof of this conjecture. 
The proof is based on the theory of Bernstein-Zelevinsky derivatives \cite{BZ1}, \cite{BZ} with the following, additional ingredient.  
The theory of derivatives describes how a smooth representation of $G_{n+1}$ restricts to the mirabolic subgroup $M_{n+1}$. However, instead of $M_{n+1}$ one can  
consider the transpose $M_{n+1}^{\top}$ of $M_{n+1}$, and develop a theory of derivatives with respect to $M_{n+1}^{\top}$. Thus we have two notions of 
derivatives: those with respect to $M_{n+1}$ are called right derivatives and those with respect to $M_{n+1}^{\top}$ are called left derivatives. 
Since $M_{n+1}^{\top}$ is not conjugated to $M_{n+1}$ in $G_{n+1}$, the information provided by left and right derivatives taken together is stronger, and is essential to 
our combinatorial arguments. Let us illustrate the argument when $\pi_1$ is the Steinberg representation of $\GL_2(F)$.  Let $\nu(g)=|g|$ be a character of $\GL_1$. 
The theory of derivatives implies that the restriction  of $\pi_1$ to $\GL_1(F)$ is given by the following Bernstein-Zelevinsky filtration 
\[ 
0 \rightarrow C_c( F^{\times}) \rightarrow \pi_1 \rightarrow \mathbb C \rightarrow 0 
\] 
where $C_c( F^{\times})$ is the space of locally constant, compactly supported functions on $F^{\times}$, and 
$\GL_1( F)$ acts on $\mathbb C$ by the character $\nu$ or $\nu^{-1}$, depending whether we use right or left derivatives, respectively. Thus, for a given character $\pi_2$ 
of $\GL_1(F)$,  one can clearly arrange that the character on the quotient $\mathbb C$ in the above sequence is different from $\pi_2$. 
Now higher extension spaces vanish since $C_c( F^{\times})$ is projective. Even multiplicity one statement is clear since it holds for $C_c( F^{\times})$. 
The general case, restricting from $G_{n+1}$ to $G_n$, follows this strategy. The bottom piece of the Bernstein-Zelevinsky filtration of $\pi_1$ is the Gelfand-Graev representation 
of $G_n$,  thus vanishing of higher extensions, and multiplicity one for generic representations, 
 follow from projectivity \cite{CS}, and multiplicity one for the Gelfand-Graev representation of $G_n$, respectively. 
 
 \smallskip 

 Let $K_r$ be the $r$-th principal congruence subgroup in $G_n$. For every $\pi\in \Alg(G_n)$ generated by the subspace $\pi^{K_r}$ of $K_r$-fixed vectors, the left ${}^{(i)}\pi$ and the right $\pi^{(i)}$ derivative are related by the 
isomorphism $(\pi^{(i)})^{\vee} \cong {}^{(i)} (\pi^{\vee})$.  This isomorphism is not true for all $\pi\in \Alg(G_n)$. 
We establish it as a consequence of a ``second adjointness isomorphism'' for Bernstein-Zelevinsky 
derivatives, proved in the appendix.  

\smallskip 

The second main result is projectivity of essentially square integrable representations of $G_{n+1}$, when restricted to $G_n$. 
If $\pi_1$ is projective as $G_n$-module then higher extension spaces vanish  without 
assuming that $\pi_2$ is generic.  On the other hand, if $\pi_2$ is degenerate, then $\EP(\pi_1, \pi_2) =0$ by the Prasad's formula. 
If $\pi_2$ is also a quotient of $\pi_1$ then $\Ext^i_{G_n}(\pi_1,\pi_2)\neq 0$ for some $i>0$. 
Thus a necessary condition for $\pi_1$ to be $G_n$-projective is not to have degenerate quotients. 
In this paper we show that this is also a sufficient condition. The proof relies heavily on the Hecke algebra methods from our earlier paper \cite{CS}. 
We also show that this condition is satisfied if $\pi_1$ is an essentially square integrable representation. 
Therefore essentially square integrable representations of $G_{n+1}$ are projective $G_n$-modules. Moreover, any two essentially square integrable representations of $G_{n+1}$ 
are isomorphic as $G_n$-modules. 
This result generalizes the classical result of Bernstein and Zelevinsky which says that any two cuspidal representations of $G_{n+1}$ are 
isomorphic when restricted to the mirabolic subgroup $M_{n+1}$ of $G_{n+1}$. We illustrate this result when $\pi_1$ is a representation of $\GL_2(F)$. If 
$\pi_1$ is cuspidal, then it is isomorphic to $C_c( F^{\times})$. On the other hand,  if $\pi_1$ is the Steinberg representation, then it is isomorphic to $C_{c}( F)$.   
The isomorphism between these two $\GL_1(F)$-modules is obtained one Bernstein component at a time, that is, for the isotypic components under the action of $O^{\times}$ where 
$O$ is the ring of integers in $F$. It is clear that, given a character of $O^{\times}$, the corresponding components are naturally isomorphic except 
when the character is trivial.  Note that $O^{\times}$-invariants are 
\[ 
C_c( F^{\times})^{O^{\times}} = \oplus_{n\in \mathbb Z} \mathbb C \varphi^{\times}_n  \text{ and } 
C_c( F)^{O^{\times}} = \oplus_{n\in \mathbb Z} \mathbb C \varphi_n 
\] 
where $\varphi^{\times}_n$  and $\varphi_n$ are the characteristic functions of $\varpi^n O^{\times}$ and $\varpi^n O$, respectively, and 
 $\varpi$ a uniformizing element.
Now $\varphi^{\times}_n \leftrightarrow \varphi_n$ gives an isomorphism. In general, for any essentially square integrable representation $\pi_1$, 
we identify each Bernstein component of $\pi_1$ with an explicit 
projective Hecke algebra module, independent of $\pi_1$. In a forthcoming paper \cite{Ch}, we shall classify all irreducible representations which are projective when restricted from $G_{n+1}$ to $G_n$, and also study the indecomposibility of irreducible smooth representations under restriction.

\smallskip 
This paper is devoted to study of the quotient restriction problem, but one can also consider the submodule restriction problem:  $\Hom_{G_n}(\pi_2,\pi_1)$. 
The two problems are related by a cohomological duality  
\[ 
\Ext_{G_n}^i(\pi_2,\pi_1)^{\vee} \cong \Ext_{G_n}^{d(\pi_2) -i}(\pi_1,  D(\pi_2)), 
\] 
where $d(\pi_2)$ is the cohomological dimension of $\pi_2$ and  $D(\pi_2)$ is the Aubert involute of $\pi_2$, 
due to Nori and Prasad \cite{NP}. This duality gives an additional importance to the cohomological 
restriction problem studied in this paper. Since $d(\pi_2)>0$, due to the presence of one-dimensional center in $G_n$, it follows that 
$\Hom_{G_n}(\pi_2,\pi_1)=0$ for all irreducible $\pi_2$ if $\pi_1$ is projective, in particular, this is true if $\pi_1$ is an essentially square integrable representation, by 
the results of this paper.

\section{Bernstein-Zelevinsky derivatives}

In this section we study Bernstein-Zelevinsky derivatives, or simply derivatives, as functors from $\Alg(G_n)$ to $\Alg(G_{n-i})$. 
We state a ``second adjointness isomorphism'' for these functors, as well as an Ext version of the formula. Mirabolic group will appear in the next section. 

\subsection{Notation} Let $G_n=\GL_n(F)$, where $F$ is $p$-adic field. Let  $\nu(g)=|\mathrm{det}(g)|$ be the character of $G_n$, where 
$|\cdot |$ is the absolute value on $F$. 
Let $B_n$ be the Borel subgroup of $G_n$ consisting of upper triangular matrices and let $U_n$ be the unipotent radical of $B_n$.  Let 
\[
R_{n-i}=  \left\{ \begin{pmatrix} g & x \\ 0 & u \end{pmatrix} : g \in G_{n-i}, u \in U_{i}, x \in \mathrm{Mat}_{n-i,i}(F) \right\} .
\]
We have an obvious Levi decomposition $R_{n-i}=G_{n-i}E_{n-i}$, where $E_{n-i}$ is the unipotent radical of $R_{n-i}$. Moreover,  
$E_{n-i}=N_{n-i}U_i$ where $N_{n-i}$ is the unipotent radical of the maximal parabolic subgroup $P_{n-i}$ consisting of block upper triangular matrices
and Levi factor $G_{n-i} \times G_i$. 
  Fix a non-zero additive character $\psi$ of $F$. Let $\psi_i$ be the character of $E_{n-i}$ defined by 
 \[ 
 \psi_i \left(  \begin{smallmatrix} 1 & x \\ 0 & u \end{smallmatrix}\right) = \psi(u_{1,2}+  \ldots  +u_{i-1,i}) 
 \] 
 where $u_{1,2}, \ldots , u_{i-1,i}$ are the entries of $u$ above the diagonal. Let $\delta_{R_i}$ be the modular character of $R_{n-i}$.  The modular character is 
 trivial on the unipotent radical $E_{n-i}$, and it is equal to $\nu^{i}$ on the Levi factor $G_{n-i}$. 
 Let $\pi$ be a smooth representation of $G_n$ on a vector space $V$.  The right $i$-th Bernstein-Zelevinsky derivative of $\pi$  is a smooth representation 
 $\pi^{(i)}$ of $G_{n-i}$ on the vector space $V^{(i)}$ defined by 
 \[  
  V^{(i)}=V / \langle \pi(e)v - \psi_i(e)v : e \in E_{n-i}, v \in V \rangle. 
\]
The representation $\pi^{(i)}$ is the natural action of the Levi factor $G_{n-i}$ on $V^{(i)}$ twisted by $\delta_{R_{n-i}}^{-1/2}$, that is, Bernstein-Zelevinsky derivatives in this paper 
are normalized. 

\smallskip 

From the definition of derivatives,  the factorization $R_{n-i}=G_{n-i}E_{n-i}$, and the Frobenius reciprocity,  one can easily prove the 
 ``first adjointness'' isomorphism for right Bernstein-Zelevinsky derivatives:   
For any smooth representation $\pi$ of $G_n$ and smooth representation $\sigma$ of  $G_{n-i}$: 
%\begin{equation} \label{E:Frobenius} 
\[ \Hom_{G_n} (\pi, \mathrm{Ind}_{R_{n-i}}^{G_n}(\sigma \otimes \psi_{i})) \cong \Hom_{G_{n-i}} (\pi^{(i)},  \sigma). \]
%\end{equation} 

We called the derivative right because there is also a left derivative, which is taken with respect to the transpose of the groups used to define right derivatives. 
More precisely, let $\theta_n(g) = (g^{-1})^{\top}$ be the outer automorphism of $G_n$ where $g^{\top}$ is the transpose of $g$  (over the usual diagonal here, but it may be over the opposite diagonal if more convenient). 
Then the left derivative of $\pi$ is defined by 
 \[  {}^{(i)} \pi = \theta_{n-i}(\theta_n(\pi)^{(i)}) .\]
In other words, the underlying vector space  for  ${}^{(i)} \pi$ is 
 \[  
  {}^{(i)}V=V / \langle \pi(e)v - \psi^{\top}_i(e)v : e \in E^{\top}_{n-i}, v \in V \rangle, 
\] 
where $\psi_i^{\top}$ is the character of $E^{\top}_{n-i}$ defined by 
 \[ 
 \psi^{\top}_i \left(  \begin{smallmatrix} 1 & 0 \\ x & u \end{smallmatrix}\right) = \bar\psi(u_{2,1}+  \ldots  +u_{i,i-1}). 
 \] 
 
 \smallskip 
 
 The following is the ``second adjointness''  isomorphism for left Bernstein-Zelevinsky derivatives, proved in the appendix. 
 
 \begin{lemma} \label{lem second adjointness bz der} Let $K_r$ be the $r$-th principal congruence subgroup in $G_n$. For any 
 representation  $\pi$ of $G_n$ generated by $\pi^{K_r}$, the space of $K_r$-fixed vectors in $\pi$,  and any smooth representation $\sigma$  of  $G_{n-i}$, 
\[ 
\Hom_{G_n} (\mathrm{ind}_{R_{n-i}}^{G_n}(\sigma \otimes \bar\psi_{i}), \pi) \cong \Hom_{G_{n-i}} (\sigma, {}^{(i)}\pi). 
\] 
\end{lemma}

\smallskip 

We now derive some consequences of the two adjointness isomorphisms. The first consequence is a relationship between right and left derivatives via the contragredient: 
 
 \begin{lemma} \label{lem_key_isomorphism} Let $K_r$ be the $r$-th principal congruence subgroup in $G_n$. 
  Let $\pi$ be a representation of $G_n$ generated by $\pi^{K_r}$. Then $(\pi^{(i)})^{\vee} \cong {}^{(i)} (\pi^{\vee})$. 
\end{lemma}
\begin{proof} If we insert $\pi^{\vee}$ in Lemma  \ref{lem second adjointness bz der} then 
\[ 
\Hom_{G_n} (\mathrm{ind}_{R_{n-i}}^{G_n}(\sigma \otimes \bar\psi_{i}), \pi^{\vee}) \cong \Hom_{G_{n-i}} (\sigma, {}^{(i)}(\pi^{\vee})). 
\] 
On the other hand,  we have three isomorphisms 
\begin{align*}
      & \Hom_{G_n} (\mathrm{ind}_{R_{n-i}}^{G_n}(\sigma \otimes \bar \psi_{i}), \pi^{\vee})   \\
\cong & \Hom_{G_n} (\pi, \mathrm{Ind}_{R_{n-i}}^{G_n}((\sigma^{\vee} \otimes  \psi_{i}))  \\
\cong & \mathrm{Hom}_{G_{n-i}}(\pi^{(i)}, \sigma^{\vee}) \\
\cong & \mathrm{Hom}_{G_{n-i}}(\sigma, (\pi^{(i)})^{\vee}) 
\end{align*}
where the first and the last are dualizing isomorphisms the second is the first adjointness isomorphism for Bernstein-Zelevinsky derivatives.  Thus 
\[ 
\mathrm{Hom}_{G_{n-i}}(\sigma, (\pi^{(i)})^{\vee}) \cong \Hom_{G_{n-i}} (\sigma, {}^{(i)}(\pi^{\vee})) 
\] 
for every smooth representation $\sigma$  of $G_{n-i}$. The lemma follows from the Yoneda lemma. 
\end{proof} 
\smallskip 
\noindent 
{\bf Remark:}  The statement of Lemma \ref{lem second adjointness bz der} is optimal in the sense that it cannot be extended to all smooth representations 
$\pi$. To that end, observe that  Lemma \ref{lem_key_isomorphism}, the case $i=n$, says that we have an isomorphism of 
vector spaces $(\pi^{\vee})_{U_n,\psi_n} \cong (\pi_{U_n, \psi_n})^{\vee}$ for every $G_n$-module $\pi$ generated by $\pi^{K_r}$. Let $\pi$ be any smooth 
$G_n$-module. It can be written as a direct sum 
\[ 
\pi\cong \oplus_{r=1}^{\infty} \pi_r 
\] 
where $\pi_r$ is generated by $\pi_r^{K_r}$ and $\pi_r^{K_{r-1}}=0$. Then 
\[ 
\pi^{\vee}\cong \oplus_{r=1}^{\infty} \pi_r^{\vee}
\] 
hence $\pi^{\vee}_{U_n,\psi_n}$ is a direct sum of  $(\pi_r^{\vee})_{U_n,\psi_n} \cong ((\pi_r)_{U_n, \psi_n})^{\vee}$. But $(\pi_{U_n,\psi_n})^{\vee}$ is a product of 
$((\pi_r)_{U_n, \psi_n})^{\vee}$, hence much larger unless the sum over $r$ is finite.

\smallskip

The following is not needed in this work, however, we cannot resist not to state it. 

\begin{lemma}  
Let $\pi$ be an irreducible representation of $G_n$. If the irreducible subquotients of ${}^{(i)}\pi$ are multiplicity free, 
then ${}^{(i)}\pi$ is a direct sum of its irreducible subquotients. 
\end{lemma}
\begin{proof} The key observation is that, in view of Lemma \ref{lem_key_isomorphism}, we have two ways to compute ${}^{(i)}\pi$: 
\[ 
{}^{(i)}\pi=  \theta_{n-i}(\theta_n(\pi)^{(i)})= ((\pi^{\vee})^{(i)})^{\vee}.
\] 
Since $\pi$ is irreducible, we have $\theta_n(\pi)\cong \pi^{\vee}$, and if we denote by $\sigma$ either of two isomorphic representations 
$\theta_n(\pi)^{(i)}$ and $(\pi^{\vee})^{(i)}$, we see that on one hand 
${}^{(i)}\pi$ is obtained from $\sigma$ by applying a co-variant functor $\theta$, and on the other hand by applying the contra-variant functor taking 
the contragradient. Since these two functors coincide on irreducible representations, corollary follows. 
\end{proof}

\begin{lemma} \label{L:sec_adj_ext} 
For any smooth representation $\pi$ of $G_n$ and smooth representation $\sigma$ of  $G_{n-i}$, 
\[ 
\Ext^j_{G_n} (\mathrm{ind}_{R_{n-i}}^{G_n}(\sigma \otimes \bar\psi_{i}), \pi) \cong \Ext^j_{G_{n-i}} (\sigma, {}^{(i)}\pi). 
\] 
\end{lemma}
\begin{proof} In order to compute the right hand side we need to use a projective resolution of $\sigma$. Using the induction in stages, 
\[ 
\mathrm{ind}_{R_{n-i}}^{G_n}(\sigma \otimes \bar\psi_{i}) \cong \mathrm{Ind}_{P_{n-i}}^{G_{n}}(\sigma \boxtimes \ind_{U_i}^{G_i}(\bar\psi_i)).
\] 
The Gelfand-Graev representation $\ind_{U_i}^{G_i}(\bar \psi_i)$ is projective by \cite{CS}. Thus, if $\sigma$ is projective it 
follows that $\mathrm{ind}_{R_{n-i}}^{G_n}(\sigma \otimes \bar\psi_{i})$ is projective, since the parabolic induction takes projective modules into projective modules. 
So we have shown that taking a projective resolution of $\sigma$ also gives a projective resolution of $\mathrm{ind}_{R_{n-i}}^{G_n}(\sigma \otimes \bar\psi_{i})$. 
Hence lemma follows from Lemma \ref{lem second adjointness bz der}. 

\end{proof}

\subsection{Zelevinsky segments} Here we follow \cite{Ze}. 
Let $\rho$ be a cuspidal representation of $G_r$. For any $a, b \in \mathbb{C}$ with $b-a \in \mathbb{Z}_{\geq0}$, we have a Zelevinsky segment 
let $\Delta=[ \nu^{a}\rho, \nu^{a+1}\rho,  \ldots, \nu^{b}\rho ]$. The absolute length of $\Delta$ is defined to be $r(b-a+1)$, and the relative $b-a+1$. 
We can truncate $\Delta$ form each side to obtain two segments of absolute length $r(b-a)$: 
\[ 
{}^{-}\Delta= [\nu^{a+1}\rho,  \ldots, \nu^{b}\rho ] \text{ and } \Delta^-= [ \nu^{a}\rho,  \ldots, \nu^{b-1}\rho ]. 
\] 
Moreover, if we perform the truncation $k$-times, the resulting segments will be denoted by ${}^{(k)}\Delta$ and $\Delta^{(k)}$. 
The induced representation 
$\nu^{a}\rho \times  \nu^{a+1}\rho \times  \ldots \times \nu^{b}\rho$  contains a unique irreducible submodule denoted by $\langle\Delta\rangle$. 

\begin{proposition} \label{P:derivative_segment} 
Let $i>0$ be an integer. The $i$-th left and right derivatives of $\langle\Delta\rangle$ vanish unless $i=r$ when 
\[ 
{}^{(r)}\langle\Delta\rangle=\langle{}^{-}\Delta\rangle \text{ and } \langle\Delta\rangle^{(r)}= \langle \Delta^-\rangle. 
\] 
\end{proposition}

\begin{corollary} \label{C:degenerate} Let $\pi$ be an irreducible subquotient of 
$\langle \Delta_1 \rangle \times \ldots \times \langle \Delta_k \rangle$.  If a right derivative of $\pi$ is generic, then every $\Delta_j$ is of the relative length one or two, and if the relative length is two, then $\Delta_j^-$ 
contributes to the cuspidal support of  the right derivative of $\pi$. Similarly, if a left derivative of $\pi$ is generic, then ${}^-\Delta_j$ contributes to the cuspidal support 
of the left derivative. 
\end{corollary} 
\begin{proof} Observe that $\langle\Delta\rangle$ is generic if and only if the relative length of $\Delta$ is one. A right derivative of 
$\langle \Delta_1 \rangle \times \ldots \times \langle \Delta_k \rangle$  has a filtration whose subquotients are 
$\langle \Delta'_1 \rangle \times \ldots \times \langle \Delta'_k \rangle$ where $\Delta_j'$ is $\Delta_j$ or $\Delta^-_j$. This representation is generic if and only if 
the relative length of every $\Delta_j$ is one or two, and if it is two then $\Delta_j'=\Delta^-_j$. 
\end{proof}

We summarize some other results from \cite{Ze} that we shall need. The induced representation 
$\nu^{a}\rho \times  \nu^{a+1}\rho \times  \ldots \times \nu^{b}\rho$  also contains a unique irreducible quotient  denoted by 
$\mathrm{St}(\Delta)$. This representation is an essentially square integrable representation i.e. its matrix coefficients are square integrable when  
restricted to the derived subgroup. Every essentially square integrable representation is isomorphic to $\St(\Delta)$ for some segment $\Delta$. 

\begin{proposition} Let $i>0$ be an integer. The $i$-th left and right derivatives of $\mathrm{St}(\Delta)$  vanish unless $i=jr$ when 
\[ 
{}^{(i)}\mathrm{St}(\Delta)=\St(\Delta^{(j)}) \text{ and } \St(\Delta)^{(i)}= \St({}^{(j)} \Delta). 
\] 
\end{proposition} 

Let $\mathfrak m=\{ \Delta_1, \ldots , \Delta_k\}$ be a multisegment i.e a multiset of segments. One say that $\mathfrak m$ is generic if 
no two segments are linked. Then $\St(\Delta_1) \times \ldots \times\St(\Delta_k)$ is an irreducible generic representation, 
and every such representation arises in this way. We shall denote this representation by $\St(\mathfrak m)$.

\section{Bernstein-Zelevinsky filtration} 
In this section we begin our study of the restriction problem from $G_{n+1}$ to $G_n$. Using the second adjointness formula, for both left and right 
derivatives, we prove that degenerate representations of $G_n$ cannot be quotients of essentially square integrable representations of $G_{n+1}$. 

\subsection{Bernstein-Zelevinsky functors} Let $M_{n+1} \subseteq G_{n+1}$ be the mirabolic subgroup 
\[
M_{n+1}= \left\{ \begin{pmatrix} g & u \\ 0 & 1 \end{pmatrix} : g \in G_{n}, u \in \mathrm{Mat}_{n,1}(F) \right\} .
\]
be the mirabolic subgroup  of $G_{n+1}$. We have an obvious Levi decomposition $M_{n+1}=G_n E_n$. Abusing notation, let  $\psi$ be the character of $E_n$ 
defined by $\psi(u)=\psi(u_n)$ where $u_n$ is the bottom entry of the column vector $u$. Note that the stabilizer of $\psi$ in $G_n$ is $M_n$. We have a pair of functors   
\[ 
\Phi^- : \Alg(M_{n+1}) \rightarrow \Alg(M_n) \text { and }  \Phi^+ : \Alg(M_{n}) \rightarrow \Alg(M_{n+1})  
\] 
defined by $\Phi^-(\tau)= \tau_{E_n,\psi}$ and $\Phi^+(\tau)=\ind_{M_n E_n}^{M_{n+1}}( \tau \boxtimes \psi)$. We also have a pair of functors 
\[ 
\Psi^- : \Alg(M_{n+1}) \rightarrow \Alg(G_n) \text { and }  \Psi^+ : \Alg(G_{n}) \rightarrow \Alg(M_{n+1})  
\] 
where $\Psi^-(\tau) = \tau_{E_n}$ and $\Psi^+$ is simply the inflation. All functors are normalized as in \cite{BZ}.  
Any $\tau\in \Alg(M_{n+1})$ has an $M_{n+1}$-filtration 
\[   \tau_n \subset  \ldots \subset \tau_0 =\tau  \]
where,  $\tau_i= (\Phi^+)^{i}(\Phi^-)^{i}(\tau)$, and 
\[ 
\tau_i/\tau_{i+1}= (\Phi^+)^{i}\Psi^+ \Psi^-(\Phi^-)^{i}(\tau).
\] 
 Observe that $\Psi^-(\Phi^-)^{i}(\tau)= \tau^{(i+1)}$,  is 
the $(i+1)$-th derivative, and the subquotients of the filtration, considered as $G_n$-modules, are 
\[  
 \tau_{i}/\tau_{i+1} \cong \mathrm{ind}_{R_{n-i}}^{G_{n}}(\nu^{1/2} \cdot \tau^{(i+1)} \boxtimes \psi_i), 
\]
where we have used notation from the previous section. In particular, $\tau_n$ is a multiple of the Gelfand-Graev representation. 
We derive some consequences of this filtration that we shall need later. 

\begin{lemma} \label{L:finite_gen} Let $\tau \in \Alg(M_{n+1})$ such that its derivatives are all finitely generated. 
When $\tau$ is considered a $G_n$-module, its Bernstein components are finitely generated. 
\end{lemma} 
\begin{proof} Recall that $P_{n-i}\supseteq R_{n-i}$ is the maximal parabolic subgroup of $G_n$ with the Levi factor $G_i\times G_{n-i}$. 
Using induction in stages, the $i$-th subquotient 
in the Bernstein-Zelevinsky filtration of $\Pi$ can be written as 
\[ 
\mathrm{Ind}_{P_{n-i}}^{G_{n}}(\nu^{1/2} \cdot \tau^{(i+1)} \boxtimes \ind_{U_i}^{G_i}(\psi_i)).
\] 
By the assumption $\tau^{(i+1)}$ is a finitely generated $G_{n-i}$-module and the Bernstein components of
the Gelfand-Graev representation $\ind_{U_i}^{G_i}(\psi_i)$ are finitely generated \cite{BH}. Lemma follows since 
parabolic induction sends finitely generated modules to finitely generated modules by 3.11 in \cite{BD}. 
%For reader's convenience we include a proof of this fact in the appendix. 
\end{proof}

\begin{lemma} \label{L:key} 
Let $\pi_1 \in \Alg(G_{n+1})$ and $\pi_2$ an admissible representation of $G_{n}$. If $\pi_2$ is a quotient of $\pi_1$ then, for some $i,j\geq 0$. 
\[ 
\Hom_{G_{n-i}}(\nu^{1/2}\cdot  \pi_1^{(i+1)}, {}^{(i)}\pi_2 ) \neq 0 \text{ and } \Hom_{G_{n-j}}(\nu^{-1/2}\cdot  {}^{(j+1)}\pi_1, \pi_2^{(j)} ) \neq 0. 
\] 
\end{lemma} 
\begin{proof} In order to prove the first isomorphism, we restrict $\pi_1$ to $G_{n}$, by way of $M_{n+1}$, and use the second adjointness formula. 
For the second we restrict to $G_{n}$, by way of $M^{\top}_{n+1}$, i.e. we reverse the roles of left and right derivatives. 
\end{proof} 

\subsection{Essentially square integrable representations} 

\begin{theorem} \label{T:St_generic} 
Let $\Delta=[\nu^a\rho, \ldots , \nu^b \rho]$ be a segment of absolute length $n+1$, where $\rho$ is a cuspidal representation of $G_r$. 
 Let $\pi$ be an irreducible $G_n$-module. If $\pi$ is a quotient of 
$\St(\Delta)$ then $\pi$ is generic. 
\end{theorem} 
\begin{proof}  Let $l=b-a+1$, in particular, $n+1=lr.$ 
Assume that $\pi$ is degenerate. Let  $\mathfrak m=\{ \Delta_1, \ldots, \Delta_k\}$ be a multi-segment, 
from the Zelevinsky classification, such that $\pi$ is the unique submodule of $\langle \Delta_1 \rangle \times \ldots \times \langle \Delta_k \rangle$. Since $\pi$ is degenerate, 
by 9.10 in \cite{Ze}  one segment in $\mathfrak m$ has the relative length at least two. 
If $\pi$ is a quotient of $\St(\Delta)$ then, by Lemma \ref{L:key}, 
${}^{(i)}\pi$ contains  $\nu^{1/2}\cdot \St(\Delta)^{(i+1)}$ as a generic submodule for some $i$. 
 Now we can apply Corollary \ref{C:degenerate}  : the relative length of each segment in $\mathfrak m$ 
is one or two, and one of them is $[\nu^{c-1/2}\rho, \nu^{c+1/2}\rho]$ where $\nu^{c+1/2}\rho$ contributes to the cuspidal support of $\nu^{1/2}\cdot \St(\Delta)^{(i+1)}$. 
It follows that $\nu^{1/2}\cdot \St(\Delta)^{(i+1)}$ is a generalized Steinberg representation corresponding to a segment ending in $\nu^{b+1/2}$, and containing 
 $\nu^{c+1/2}\rho$.  Thus, for every  $d=c, \ldots, b$, $\nu^{d+1/2}\rho$ contributes to the cuspidal support of ${}^{(i)}\pi$ as well as to the cuspidal support of $\pi$. Similarly, if we use the second identity in Lemma \ref{L:key}, then for every $d= a, \ldots, c$, $\nu^{d-1/2}\rho$ contributes to the cuspidal support of $\pi$. We see that $\mathfrak m$ contains segments of total relative length $\geq l$ and absolute 
 length $(l+1)r=n+1+r >n$.  
This is a contradiction. 
\end{proof}

\section{Vanishing of Ext's} 

The purpose of this section is to prove: 

\begin{theorem} 
Let $\pi_1$ be an irreducible generic representation of $G_{n+1}$ and $\pi_2$ an irreducible generic representation of $G_{n}$. Then 
\[ 
\Ext_{G_n}^i( \pi_1,  \pi_2)=0 \text{ if } i >0 \text{ and } \cong \mathbb C \text{ if } i=0. 
\] 
\end{theorem} 

Let us explain the strategy of the proof. Fix $\pi_2$, and assume that 
$\pi_2$ is a subquotient of $\rho_1\times \ldots \times \rho_k$ where $\rho_i$ are cuspidal representations. 
Let $m(\pi_1)$ be the integer that counts the number of cuspidal representations $\rho$ in the support of $\pi_1$ such that $\rho$ is an unramified 
twist of a $\rho_i$, for some $1\leq i \leq k$. The proof is by induction on $m(\pi_1)$. The base 
case $m(\pi_1)=0$ is easy. It is deduced from the Bernstein-Zelevinsky filtration of $\pi_1$ where the bottom piece is the Gelfand-Graev representation of $G_n$. 
Assume now that $\pi_1=\St(\mathfrak m_1)$ and $\pi_2=\St(\mathfrak m_2)$ for a pair of generic multisegments $\mathfrak{m}_1$ and $\mathfrak{m}_2$ i.e. no two segments in $\mathfrak{m}_i$ are linked. 
 Let $\Delta=[\nu^a\rho, \ldots , \nu^b\rho]$ be a 
segment in $\mathfrak m_1$ such that $\rho$ contributes to $m(\pi_2)$.  Assume that $\Delta$ is also a shortest such segment. 
Write $\pi_1=\St(\Delta) \times \pi$ where $\pi=\St(\mathfrak m)$ and 
$\mathfrak m=\mathfrak m_1 \setminus \Delta$. 
Let $r$ be the integer such that  $\rho \in \Alg(G_r)$. Let $\rho'\in \Alg(G_r)$ be another cuspidal representation such that no unramified twist of $\rho'$ appears in 
the cuspidal supports of $\pi_1$ and $\pi_2$. Now both $\rho'\times \St({}^{-}\Delta) \times \pi$ and $\rho'\times \St(\Delta^-) \times \pi \in \Alg(G_{n+1})$ are irreducible 
 and satisfy the induction assumption.  We shall use this information to prove the theorem for $\pi_1$.

\subsection{Transfer}

Let $l=s+r$. Recall that $P_s$ is the maximal parabolic of $G_l$ with the Levi $G_s\times G_{r}$. 
Starting with $\sigma \in \Alg(G_s)$ and $\tau\in \Alg(M_r)$ one can manufacture two representations of $M_l$. 
The first one is obtained by the (normalized) induction from $P_s\cap M_l$ and, abusing notation, denoted by $\sigma\times \tau$. 
The second is obtained by the normalized induction from 
$P_s ^{\top}\cap M_l$ but only after $\sigma$ is multiplied by $\nu^{-1/2}$, see \cite{BZ} page 457, where the definition uses a different subgroup, but conjugated in $M_l$.  
This representation is denoted by $\tau\times  \sigma$.  

Our interest in these representations comes from the following,   4.13 Proposition in \cite{BZ}. 

 \begin{proposition} \label{P:BZ} 
 Let $\rho\in \Alg(G_r)$, $\sigma \in \Alg(G_s)$ and $\tau\in \Alg(M_r)$. Let $\rho\vert_M$ and $\sigma\vert_M$ denote restrictions to $M_r$ and $M_s$, respectively. 
 \begin{enumerate} 
\item  There exists an exact sequence in $\Alg(M_l)$ 
 \[ 
 0\rightarrow (\rho\vert_{M}) \times  \sigma \rightarrow \rho \times \sigma   \rightarrow \rho \times (\sigma \vert_{M}) \rightarrow 0
 \] 
 \item If $\Omega$ is any of the four functors $\Phi^{\pm}$ and $\Psi^{\pm}$, then 
\[ 
\Omega(\sigma\times \tau)=\sigma\times \Omega(\tau). 
\] 
\item $\Psi^-(\tau\times \sigma)= \Psi^-(\tau) \times \sigma $, and 
there exists an exact sequence in $\Alg(M_{l-1})$ 
 \[ 
 0\rightarrow \Phi^-(\tau)\times \sigma \rightarrow \Phi^-(\tau\times \sigma) \rightarrow \Psi^-(\tau) \times (\sigma \vert_{M}) \rightarrow 0
 \] 
\end{enumerate} 
 \end{proposition}

 \begin{proposition} \label{P:transfer} 
 Let $\Delta=[\nu^a\rho, \ldots, \nu^b\rho]$ be a segment where $\rho\in \Alg(G_r)$. 
 Let $\tau_r=(\Phi^+)^{r-1}(1) \in \Alg(M_r)$, the Gelfand-Graev module. 
 Then $\St(\Delta)\vert_{M}$ is isomorphic to $\tau_r \times \St({}^-\Delta)$.  
 \end{proposition} 
 \begin{proof}  
 Recall that $\rho\vert_{M}\cong \tau_r$ (this is true for every cuspidal representation). 
 Note that  $\St(\Delta)$ is a quotient of $\nu^a \rho \times \St({}^-\Delta)$. By Proposition \ref{P:BZ} (1),  we have an exact sequence of mirabolic subgroup modules 
  \[ 
 0\rightarrow \tau_r \times \St({}^-\Delta) \rightarrow \nu^a \rho \times \St({}^-\Delta)    \rightarrow \nu^a\rho \times (\St({}^-\Delta) \vert_{M}) \rightarrow 0
 \] 
 By Proposition \ref{P:BZ} (2), any derivative of  the quotient in the above sequence is equal to $\nu^a\rho \times \St({}^{(k)}\Delta)$  with $k>1$. 
   Since $\nu^a\rho$ and ${}^{(k)}\Delta$  are not linked, the corresponding subquotients in 
 the Bernstein-Zelevinsky filtration are irreducible as mirabolic subgroup modules. Observe that they are non-isomorphic to the subquotients of the Bernstein-Zelevinsky filtration of $\St(\Delta)$. 
 Hence the projection from  $\nu^a \rho \times \St({}^-\Delta)$  onto $\St(\Delta)$, restricted to $\tau_r \times \St({}^-\Delta)$ gives the desired 
 isomorphism. 
 \end{proof} 
 
 Now we arrive to a key result: 
 
 \begin{corollary} \label{C:transfer} Let $\rho, \rho'\in \Alg(G_r)$ be two irreducible cuspidal representations. 
 Let $\Delta=[\nu^a\rho, \ldots, \nu^b\rho]$, and $\pi\in \Alg(G_s)$.  Then we have an isomorphism of mirabolic modules 
 \[ 
 \St(\Delta)\vert_{M} \times \pi \cong \rho'\vert_{M} \times (\St({}^-\Delta)\times \pi). 
 \] 
 \end{corollary} 
 \begin{proof} By Proposition \ref{P:transfer}, we can substitute $\St(\Delta)\vert_{M}= \tau_r \times \St({}^-\Delta)$. Furthermore, we have a natural isomorphism 
 \[ 
(\tau_r \times \St({}^-\Delta)) \times \pi \cong \tau_r \times (\St({}^-\Delta)) \times \pi) 
\] 
given by the induction in stages in two different orders. 
We finish by observing that $\tau_r= \rho'\vert_{M}$.  
 \end{proof} 

 Now we continue with the proof of vanishing for $\pi_1=\St(\Delta)\times \pi$, notation as in the start of the section. By  Proposition \ref{P:BZ} (1) there is an exact sequence 
 in $\Alg(M_{n+1})$ 
 \[ 
 0\rightarrow (\St(\Delta)\vert_{M}) \times  \pi  \rightarrow \St(\Delta)\times \pi   \rightarrow \St(\Delta)\times (\pi \vert_{M}) \rightarrow 0. 
 \] 
 Likewise, there is an exact sequence  in $\Alg(M_{n+1})$ 
 \[ 
 0\rightarrow \rho'\vert_{M} \times  (\St({}^-\Delta)\times \pi)  \rightarrow \rho'\times (\St({}^-\Delta)\times \pi)   \rightarrow \rho'\times( \St({}^-\Delta)\times \pi )\vert_{M} \rightarrow 0. 
 \] 
Note that the submodules in the two sequences are isomorphic by Corollary \ref{C:transfer}. Furthermore, by the choice of $\rho'$, 
\[ 
\Ext_{G_n}^i(  \rho'\times( \St({}^-\Delta)\times \pi )\vert_{M} , \pi_2)=0 \text{ if } i\geq 0. 
\] 
Now  we can apply the induction assumption to $\rho'\times \St({}^-\Delta)\times \pi$ and conclude that 
\[ 
\Ext_{G_n}^i((\St(\Delta)\vert_{M}) \times  \pi , \pi_2)=0 \text{ if } i >0 \text{ and } \cong \mathbb C \text{ if } i=0. 
\] 
Hence, in order to establish the conjecture for the pair $(\pi_1,\pi_2)$, it suffices to show that 
\[ 
\Ext_{G_n}^i( \St(\Delta)\times (\pi \vert_{M}), \pi_2)=0 \text{ if } i \geq 0, 
\] 
and to do this it suffices to show vanishing for  each subquotient in the Bernstein-Zelevinsky filtration of $\St(\Delta)\times (\pi \vert_{M})$. 
By Proposition \ref{P:BZ} part (2), the derivatives of $\St(\Delta)\times (\pi \vert_{M})$ are computed on the second factor. Thus, combining with 
the second adjointness formula, it suffices to show that 
\begin{itemize} 
\item $\Ext_{G_n}^j( \nu^{1/2}\St(\Delta)\times \pi^{(i+1)}, {}^{(i)}\pi_2)=0$ for $i,j\geq 0$. 
\end{itemize} 
Alternatively, by reversing the roles of left and right derivatives,  
it suffices to show that 
\begin{itemize} 
\item $\Ext_{G_n}^j( \nu^{-1/2}\St(\Delta)\times {}^{(i+1)}\pi , \pi^{(i)}_2)=0$ for $i,j\geq 0$. 
\end{itemize}

\smallskip

Hence it suffices to show that the Bernstein center spectra of  $\nu^{1/2}( \St(\Delta) \times \pi^{(i+1)})$ and of ${}^{(i)}\pi_2$ are different for all $i$, or, 
they are different for $\nu^{-1/2}( \St(\Delta) \times {}^{(i+1)} \pi)$ and ${}\pi_2^{(i)}$ for all $i$. The strategy is to show that, if both statements fail, then
$\mathfrak m_2$ contains linked segments. 

\subsection{Combinatorics} 
 Let $\mathfrak m=\{ \Delta_1, \ldots , \Delta_k\}$ be a multisegment. Then $\St(\mathfrak m)$ is generic but reducible if some segments are linked. 
 However, if $\Delta_i$ and $\Delta_j$ are linked, then they can be replaced by $\Delta_i \cap \Delta_j$ and $\Delta_i\cup \Delta_j$. This process (called 
 recombination henceforth) eventually leads to a generic segment such that the corresponding irreducible generic representation is the unique generic 
 subquotient in $\St(\mathfrak m)$. Important observation is that the points in the spectrum of the Bernstein center are uniquely represented by generic multisegments! 
  The following is a key lemma. 

\begin{lemma} 
Let $\mathfrak m$  be a generic multisegment and $\mathfrak m'$  a multisegment obtained by truncating $\mathfrak m$ from the right. 
 Then the generic segment 
corresponding to $\mathfrak m'$ by recombination is also obtained from $\mathfrak m$ by truncating from the right. 
\end{lemma} 
\begin{proof} This is proved by induction on the number of steps in the recombination process. If that number is 0 there is nothing to prove. Otherwise there is 
a pair of linked segments  $\Delta'$ and $\Delta''$ in $\mathfrak m'$ such that the first step in the recombination is replacing  
$\Delta'$ and $\Delta''$ by $\Delta' \cap \Delta''$ and $\Delta' \cup \Delta''$. It is trivial to see that the resulting multisegment is also obtained by right truncation from $\mathfrak m$. 
 The proof follows by induction. 
\end{proof}

\subsection{Finishing the proof}  Let $l=b-a+1$ be the relative length of $\Delta$. 
  We note that  ${}^{(i)}\pi_2$ is glued from $\St(\mathfrak m'_2)$ where $\mathfrak m'_2$ runs over all multisegments obtained from $\mathfrak m_2$ by truncating from the right 
  $i$-times (in the sense of absolute length). By the previous lemma the Bernstein center spectrum of ${}^{(i)}\pi_2$ is given by such generic multisegments. 
   Likewise,   $\St(\Delta)\times  \pi^{(i+1)}$  is glued from $\St(\Delta) \times \St({}'\mathfrak m)$ where ${}'\mathfrak m$ runs over all multisegments obtained from $\mathfrak m$ by truncating from the left $i+1$-times, and to determine the Bernstein center spectrum  we need to consider 
only  generic ${}'\mathfrak m$. However, $\{\Delta\}  \cup {}'\mathfrak m$ needs not be generic. There could be segments in 
${}'\mathfrak m$ linked to $\Delta$. Since $\Delta$ is not linked to any segment in $\mathfrak m$ and ${}'\mathfrak m$ is obtained  from $\mathfrak m$ by 
left truncation, it follows that linking occurs over the right end point of $\Delta$. 
Let $\Delta_0$ be the longest segment in ${}'\mathfrak m$ linked 
to $\Delta$. It is easy to see that $\Delta \cup \Delta_0$ is a segment in the generic multisegment corresponding to $\{\Delta\}  \cup {}'\mathfrak m$ by the 
recombination process.  Note that $\Delta \cup \Delta_0$ starts with $\nu^a\rho$ and is of relative length at least $l$. 
 Thus the Bernstein spectra of $\nu^{1/2} (\St(\Delta)\times  \pi^{(i+1)})$ and ${}^{(i)}\pi_2$ 
can have a point in common only if $\mathfrak m_2$ contains a segment starting with $\nu^{a+1/2}\rho$ and of relative length at least $l$. Similarly, 
the Bernstein spectra of $\nu^{-1/2}( \St(\Delta)\times  {}^{(i+1)}\pi)$ and ${}\pi_2^{(i)}$  can have a point in common only if $\mathfrak m_2$ contains a segment 
ending with $\nu^{b-1/2}\rho$ and of length at least $l$. In other words we have constructed a pair of linked segments in $\mathfrak m_2$, a contradiction. This completes the proof of the Ext-vanishing theorem. 

\section{Hecke algebra methods} 

The main goal of this section is to prove projectivity of an essentially square integrable representation $\pi_1$ of $G_{n+1}$ when restricted to $G_n$. The proof uses Hecke algebras 
and identifies all Bernstein components of $\pi_1$ with the sign-projective module of the Hecke algebra corresponding to the Bushnell-Kutzko type \cite{BK1}, \cite{BK2}, \cite{BK3}. 
 As a consequence, any two essentially square integrable representations of $G_{n+1}$ are isomorphic when restricted to $G_n$. 

\subsection{Hecke algebras}
Let $\Delta=[ \nu^a\rho, \ldots, \nu^b\rho ]$ be a Zelevinsky segment. Let $m=b-a+1$. The Bernstein component of $\St(\Delta)$ is equivalent to the category of representations 
of a Hecke algebra  $\mathcal H_m$ arising from a simple Bushnell-Kutzko type $\tau_{\Delta}$, that is, if $\pi$ is a smooth representation in the Bernstein component, 
then $\Hom(\tau_{\Delta}, \pi)$ is the corresponding $\mathcal H_m$-module. 
 The algebra $\mathcal H_m$ is isomorphic to the Iwahori Hecke algebra of $\GL_m(E)$, for some field $E$.  Thus, as an abstract algebra, 
 $\mathcal H_m$ is generated  by $\theta_1, \ldots, \theta_{m}$,  and $T_w$ ($w \in S_m$) satisfying the following relations:
\begin{enumerate}
\item $\theta_k\theta_l=\theta_l\theta_k$ for any $k,l=1,\ldots, m$;
\item $T_{s_k}\theta_k-\theta_{k+1}T_{s_k}=(q-1)\theta_{k}$, where $q$ is a prime power depending on $\tau_{\Delta}$ and $s_k$ is the transposition of numbers $k$ and $k+1$;
\item $T_{s_k}\theta_l =\theta_lT_{s_k}$ if $l \neq k, k+1$;
\item $(T_{s_k}-q)(T_{s_k}+1)=0$, where $s_k$ is as in (2), and $T_w$ satisfies a braid relation.
\end{enumerate}
Let $\mathcal A_m=\mathbb C[\theta_1^{\pm 1}, \ldots , \theta_m^{\pm 1}]$ and $\mathcal H_{S_m}$ be the span of $T_w$, $w\in S_m$. 
Then $\mathcal H_m \cong \mathcal A_m \otimes \mathcal H_{S_m}$.  The finite dimensional algebra $\mathcal H_{S_m}$  has a one dimensional 
sign representation $\mathrm{sgn}(T_w)= (-1)^{\ell(w)}$, where $\ell$ is the length function on $S_m$. An irreducible representation $\pi$ in the component is 
Whittaker generic if and only if $\Hom(\tau_{\Delta},\pi)$ contains the sign type as $\mathcal H_{S_m}$-module \cite{CS}. 

\smallskip 

Let $\Delta_1, \ldots, \Delta_r$ be segments such that  for $i\neq j$, the cuspidal representations $\rho_i$ and $\rho_j$ are not unramified twists of each other.  
The Bernstein component of $\St(\Delta_1)\times \cdots \times \St(\Delta_r)$ is equivalent to the category of representations 
of a Hecke algebra  $\mathcal H$ arising from a semi-simple Bushnell-Kutzko type $\tau$.  We have 
$\mathcal H\cong \mathcal H_{m_1} \otimes \cdots \otimes \mathcal H_{m_r}$ and $\mathcal H \cong \mathcal A \otimes \mathcal H_{S}$ where 
$\mathcal A\cong  \mathcal A_{m_1} \otimes \cdots  \otimes \mathcal A_{m_r}$ and $\mathcal H_{S}\cong \mathcal H_{S_{m_1}} \otimes \cdots \otimes  \mathcal H_{S_{m_r}}$. 
The subalgebra $\mathcal A$ is isomorphic to the ring of Laurent polynomials in $m=m_1+ \ldots + m_r$ variables, while $\mathcal H_{S}$ is spanned by $T_w$, 
$w\in S=S_{m_1} \times \cdots \times S_{m_r}$. An irreducible representation $\pi$ in the component can be written as $\pi_1\times \cdots \times \pi_r$ where 
$\pi_i$ is in the component of $\St(\Delta_i)$, thus it clear that $\pi$ is Whittaker generic if and only if $\Hom(\tau, \pi)$ contains the sign type of $\mathcal H_S$. 

\subsection{Some projective modules} 
Let $\chi$ be a character of $\mathcal A$. 
The $\mathcal H$-module $\mathcal H \otimes_{\mathcal A} \chi$ is called the principal series representation of $\mathcal H$. A twisted Steinberg representation 
of $\mathcal H$ is any one-dimensional $\mathcal H$-module such that the restriction to $\mathcal H_S$ is the sign type. For example, if $\pi= \St(\Delta_1)\times \cdots \times \St(\Delta_r)$, 
then $\Hom(\tau, \pi) $ is a twisted Steinberg representation. 

The following  is Theorem 2.1 in \cite{CS}. (It is stated there for $\mathcal H$ arising from the singleton partition $(m)$ but the proof is applicable to a general 
partition $(m_1, \ldots , m_r)$). 

\begin{theorem} \label{T:IMRN}
Let $\Pi$ be an $\mathcal H$-module. Assume that 
\begin{enumerate} 
\item $\Pi$  is projective and finitely generated.
\item $\dim \Hom_{\mathcal H}(\Pi, \pi)  \leq 1$ for an irreducible principal series representation $\pi$. 
\item A twisted Steinberg representation is a quotient of $\Pi$. 
\end{enumerate} 
Then $\Pi \cong \mathcal H\otimes_{\mathcal H_S} \mathrm{sgn}$. 
\end{theorem}

As in \cite{CS}, we have the following corollary. 

\begin{corollary}  \label{C:GG} 
Let $\Gamma$ be the summand of the Gelfand-Graev representation corresponding to the Bernstein component of 
$\St(\Delta_1)\times \cdots \times \St(\Delta_r)$. Then we have an isomorphism 
$\Hom(\tau, \Gamma) \cong \mathcal H\otimes_{\mathcal H_S} \mathrm{sgn}$ of $\mathcal H$-modules. 
\end{corollary} 

\subsection{Projectivity for Hecke algebras} 

Let $\mathcal Z$ be the center of $\mathcal H$.  Recall that $\mathcal Z=\mathcal A^S$, in particular, $\mathcal H$ is a finitely generated $\mathcal Z$-module. 
Let $\mathcal J$ be a maximal ideal in $\mathcal Z$. 
Let $\hat{\mathcal H}$ denote the $\mathcal J$-adic completion of $\mathcal H$ \cite{AMD}. For every $\mathcal H$-module $\Pi$, let $\hat\Pi$ denote the 
$\mathcal J$-adic completion of $\Pi$. If $\Pi$ is finitely generated, then $\hat \Pi\cong \hat{\mathcal H} \otimes_{\mathcal H} \Pi$. 

\begin{theorem} \label{T:local_projectivity} 
 Let $\Pi$ be a finitely generated $\mathcal H$-module and  $\mathcal J$ a maximal ideal in $\mathcal Z$.  
 Let $\pi$ be the unique irreducible $\mathcal H$-module annihilated by $\mathcal J$ and containing the sign type. Assume that 
\begin{enumerate} 
\item $\dim \Hom_{\mathcal H}(\Pi, \pi) = 1$ 
\item $\Pi$ has no other irreducible quotients annihilated by $\mathcal J$. 
\item $\Pi$ contains a torsion free element for $\mathcal A$. 
\end{enumerate} 
Then $\hat\Pi \cong \hat{\mathcal H}\otimes_{\mathcal H_S} \mathrm{sgn}$. 
\end{theorem} 
\begin{proof}  In order to simplify notation, write $\Sigma=\mathcal H\otimes_{\mathcal H_S} \mathrm{sgn}$. 
Since $\Pi$ is finitely generated, $\hat{\Pi}/\mathcal J\hat{\Pi} \cong \Pi/\mathcal J\Pi$ is a finite dimensional $\mathcal H$-module, annihilated by $\mathcal J$. 
By (2) it must be generated by the sign type subspace. Let $r$ be the dimension of the sign type in $\Pi/\mathcal J\Pi$. By Frobenius reciprocity, we have a surjection 
$f:\Sigma^{\oplus r} \rightarrow \Pi/\mathcal J\Pi$ which descends to a surjection $\bar f:(\Sigma/\mathcal J \Sigma)^{\oplus r} \rightarrow \Pi/\mathcal J\Pi$. Observe that $\bar f$ is bijective 
on the sign type, since the sign type in $\Sigma/\mathcal J\Sigma$ is one-dimensional. Since 
 $\pi$ is the unique irreducible quotient of $\Sigma/\mathcal J\Sigma$ and $\bar f$ is bijective on the sign type, 
 it follows that $\pi^r$ is a quotient of $\Pi/\mathcal J\Pi$. This forces $r=1$  by (1) and, 
 by Nakayama lemma, we have a surjection $\hat f: \hat{\Sigma} \rightarrow \hat{\Pi}$.     
Since $\hat{\Sigma}\cong  \hat{\mathcal A}$, as $\hat{\mathcal A}$-modules, (3) implies that the surjection is in fact an 
isomorphism. 

\end{proof} 

\begin{corollary} \label{C:local_vanishing} 
 Let $\Pi$ be a finitely generated $\mathcal H$-module and $\mathcal J$ a maximal ideal in $\mathcal Z$. 
Assume that the conditions of Theorem \ref{T:local_projectivity}  are satisfied. Then, for all $\mathcal H$-modules $\sigma$ annihilated by $\mathcal J$ and for all $i>0$, 
\[ 
\mathrm{Ext}^i_{\mathcal H} (\Pi, \sigma) =0. 
\] 
\end{corollary} 
\begin{proof}  Since $\sigma$ is annihilated by $\mathcal J$, we have 
\[ 
\mathrm{Ext}^i_{\mathcal H} (\Pi, \sigma) \cong \mathrm{Ext}^i_{\hat{\mathcal H}} (\hat{\Pi}, \sigma). 
\] 
The latter spaces are trivial by projectivity of $\hat{\mathcal H}\otimes_{\mathcal H_S} \mathrm{sgn}$.  
\end{proof} 

\begin{corollary} \label{C:global_vanishing} 
 Let $\Pi$ be a finitely generated $\mathcal H$-module. 
Assume that the conditions of Theorem \ref{T:local_projectivity} are satisfied for every maximal ideal in $\mathcal Z$. 
Then $\Pi \cong \mathcal H\otimes_{\mathcal H_S} \mathrm{sgn}$.
\end{corollary} 
\begin{proof}  Corollary \ref{C:local_vanishing} implies that $\mathrm{Ext}^i_{\mathcal H} (\Pi, \sigma) =0$  for all finite length modules. 
Since $\Pi$ is also finitely generated, it is projective by Theorem 8.1 in the appendix of \cite{CS}.  Now we can apply Theorem \ref{T:IMRN}. 
\end{proof} 

\subsection{Projectivity for groups} 
Now we can apply the Hecke-module results to the restriction problem, one Bernstein component at the time. Let $\pi_1$ be an irreducible generic representation of 
$G_{n+1}$ and fix a Bushnell-Kutzko type $\tau$ for $G_n$. Let $\Pi= \Hom(\tau, \pi_1)$ be the corresponding $\mathcal H$-module. Note that the conditions (1) and (3) 
in Theorem \ref{T:local_projectivity} are satisfied for every maximal ideal $\mathcal J$. Indeed, the condition (1) because all irreducible generic $G_n$-representations are 
quotients of $\pi_1$ with multiplicity one and (3) because $\pi_1$,  restricted to $G_n$,  contains 
the Gelfand-Graev representation, a free $\mathcal A$-module.  Theorem \ref{T:local_projectivity} implies the following local Ext vanishing result for groups. 

\begin{theorem} \label{T:vanish_local} Let $\pi_1$ be an irreducible generic representation of $G_{n+1}$. Let $\mathcal J$ be a maximal ideal of the Bernstein center of 
$G_n$. Assume that no degenerate irreducible representation of $G_n$ annihilated by $\mathcal J$ is a quotient of $\pi_1$. Then 
$\mathrm{Ext}^i_{G_n} (\pi_1, \pi_2)=0$, $i>0$, for all irreducible representation $\pi_2$ of $G_n$ annihilated by $\mathcal J$. 
\end{theorem}

Finally since, by Theorem \ref{T:St_generic}, essentially square integrable representation have no degenerate quotients, Corollary \ref{C:global_vanishing} implies: 

\begin{theorem}\label{T:projective_general} 
 Let $\pi_1$ be an essentially square integrable representation of $G_{n+1}$. Then  $\pi_1$ considered a $G_n$-module, is projective. Moreover, 
if $\pi'_1$ is another essentially square integrable representation of $G_{n+1}$ then  $\pi_1$ and $\pi'_1$ are isomorphic as $G_n$-modules. 
\end{theorem}

\section{Appendix} 

In this appendix we prove Lemma \ref{lem second adjointness bz der}, that is, the second adjointness isomorphism for Bernstein-Zelevisky derivatives. The key 
ingredient is Rodier's approximation \cite{Ro} of the Whittaker character by characters of compact pro-$p$ groups.

\subsection{Groups} Let $F$ be a $p$-adic field, $R$ its ring of integers and $P$ the maximal ideal generated by a prime $\varpi$.
Let $\psi$ be the character of $F$ of conductor $R$. Let $G=\GL_n(F)$ and let 
$U$ be the group of unipotent upper triangular matrices in $G$. 
Let $\psi_U: U \rightarrow \mathbb C$ be a Whittaker character defined by 
\[
\psi_U(u)=\psi(u_{1,2} + \cdots +u_{n-1,n})
\] 
 where $u_{i,j}$ denote the entries of the matrix $u$. 
 
 \smallskip 
For every natural number $r$, 
let $L_r$ be the lattice in $M_n(F)$ consisting of all matrices whose entries are in $P^r$. Then 
\[ 
K_r = 1 + L_r
\] 
is a principal congruence subgroup of $G$. Let $t=(t_i)\in G$ be a diagonal matrix such that $t_i/t_{i+1}=\varpi^2$ for $i=1,\ldots, n-1$. 
 Let $H_r=t^{-r} K_r t^{r}$.  Let  $B^{\top}$ be the Borel 
subgroup of lower-triangular matrices. Then we have a parhoric decomposition  
\[ 
H_r= (H_r\cap B^{\top}) (H_r\cap U). 
\] 
The sequence of groups $H_r\cap B^{\top}$ is decreasing with trivial intersection, while the sequence of groups $H_r\cap U$ is increasing with union $U$. 
 Let $\psi_r$ be a character of $H_r$ defined by 
\[ 
\psi_r(g) = \psi( g_{1,2} + \cdots +g_{n-1,n}). 
\]
Observe that 
\[ 
\psi_r|_{H_r \cap U} = \psi_U|_{H_r\cap U}. 
\]

\subsection{Representations}  Let  $\pi$ be a smooth $G$-module.   For every non-negative integer $r$ we have a projection map $P_r : \pi \rightarrow \pi^{H_r, \psi_r}$ defined 
by 
\[ 
P_r(v) = \mathrm{vol}(H_r)^{-1} \int_{H_r} \bar\psi_r(u) \pi(g) v ~dg.
\]  
 
 For $r\leq s $ we have maps $i_r^s : \pi^{H_r, \psi_r} \rightarrow \pi^{H_s, \psi_s}$ defined by restricting $P_s$ to $\pi^{H_r, \psi_r}$. From the parahoric decomposition of 
 $H_r$, it is easy to see that 
\[ 
i_r^s(v) =  \mathrm{vol}(H_s\cap U)^{-1}\int_{H_s\cap U} \bar\psi_s(u) \pi(u) v ~du. 
\] 
This formula, in turn, implies that these maps form a direct system i.e. $i_s^t\circ i_r^s = i_r^t$, for $r \leq s \leq t$.  
We have natural maps $i_r : \pi^{H_r, \psi_r} \rightarrow \pi_{U,\psi_U}$.  
Observe that $i_s\circ i_r^s = i_r$. Hence we have a  map 
\[ 
i_\pi : \lim_r  \pi^{H_r, \psi_r} \rightarrow \pi_{U,\psi_U}.
\]

\begin{proposition} \label{P:compact_approximation}  For every smooth $G$-module $\pi$ the map $i_{\pi}$ is an isomorphism of vector spaces. 
\end{proposition} 
\begin{proof}  Surjectivity: Let $v\in \pi$.  Since  $H_r\cap B^{\top} \rightarrow \{1\}$
there exists $r$ such that $v$ is $H_r\cap B^{\top}$-invariant.
 Let 
\[ 
w=\mathrm{vol}(H_r\cap U)^{-1}  \int_{H_r\cap U} \bar\psi_r(u) \pi(u) v ~du \in \pi^{H_r,\psi_r}. 
\] 
Then $v$ and $w$ have the same projection on $\pi_{U,\psi_U}$.  Injectivity: Let $v\in \pi^{H_r,\psi_r}$ that projects to $0$ in $\pi_{U,\psi_U}$. Then  there exists an 
open compact subgroup $U_c \subset U$ such that 
\[ 
\int_{U_c} \bar\psi_s(u) \pi(u) v ~du =0. 
\] 
Since $H_s\cap U\rightarrow U$ there exists $s\geq r$ such that $ H_s\cap U \supset U_c$. Then the above integral, with $U_c$ substituted by $H_s\cap U$, vanishes. 
In other words, $i_r^s(v)=0$ and hence $v=0$, viewed as an element of the direct limit. 
\end{proof} 

For $r\leq s$ we have maps $p_r^s : \pi^{H_s, \psi_s} \rightarrow \pi^{H_r, \psi_r}$,  going in the opposite direction, defined by restricting 
$P_r$ to $\pi^{H_s \psi_s}$. From the parahoric decomposition of 
 $H_r$, it is easy to see that 
\[ 
 p_r^s(v) = \mathrm{vol}(H_r\cap B^{\top})^{-1}\int_{H_r\cap B^{\top}} \pi(g) v ~dg 
\] 
and this implies that these maps form an inverse  system i.e. $p_r^s\circ p_s^t = p_r^t$, for $r \leq s \leq t$. 

\smallskip 
 By Proposition 4 in  \cite{Ro}  (see also VI, page 169), there exists an integer $r_0$, independent of 
$\pi$, such that $p_r^s \circ i_r^s$ is a non-zero multiple of the identity on $\pi^{H_r,\psi_r}$, if $r_0 \leq r \leq s$. 
 Thus $i_r^s$ is an injection and $p_s^r$ is a surjection.  
 It follows, from Proposition \ref{P:compact_approximation}, that the maps $i_r : \pi^{H_r,\psi_r} \rightarrow \pi_{U,\psi_U}$ are injections, for all $r\geq r_0$. 
 
 \smallskip 
 We shall use surjectivity of the maps $p_r^s$ to construct a natural complement of $\pi^{H_r,\psi_r}$ in $\pi_{U,\psi_U}$. So fix $r\geq r_0$, and for every $s\geq r$, let 
$\tau_s$ be the kernel of $p_r^s$. Observe that $\tau_s$ is a complement of $\pi^{H_r,\psi_r}$ in $\pi^{H_s,\psi_s}$, where we have identified $\pi^{H_r,\psi_r}$ with its 
image in $\pi^{H_s,\psi_s}$. 
We claim that $\tau_s$, for $s\geq r$ form an injective subsystem. To that end, let $t\geq s$. We need to prove if $v\in \tau_s$ then 
$i_s^t(v) \in \tau_t$, that is, $p_r^t (i_s^t (v))=0$. Write $p_r^t = p_r^s  \circ p_s^t $. Then 
\[ 
p_r^t (i_s^t (v))= p_r^s  \circ p_s^t ( (i_s^t (v))=  p_r^s ( p_s^t  \circ i_s^t (v)) = 0 
\] 
where for the last equality we used that $p_s^t  \circ i_s^t (v)$ is a multiple of $v$. 
Hence 
\[ 
\pi^{H_r,\psi_r}_c := \lim_{s\geq r} \tau_s
\] 
 is a complement of $\pi^{H_r,\psi_r}$ in $\lim_{s\geq r}\pi^{H_r,\psi_r} \cong \pi_{U,\psi_U}$.

\smallskip 

We apply the above considerations to $\pi=S(G)$, the space of locally constant, compactly supported functions on $G$, considered a $G$-module with respect to the action by 
left translations. In this case, the vector spaces $\pi^{H_r,\psi_r}$ and $\pi_{U,\psi_U}$  are naturally $G$-modules, coming from the right translation action of $G$ on $S(G)$ and the maps $i_r^s$, $p_r^s$ and $i_r$ are $G$-morphisms. Observe that 
$S(G)^{H_r,\psi_r}= \ind_{H_r}^G(\psi_r)$ and 
$S(G)_{U,\psi_U}\cong \ind_{U}^G(\psi)$, the Gelfand-Graev representation.  
Hence $\lim_r \ind_{H_r}^G(\psi_r) \cong \ind_{U}^G(\psi)$, as $G$-modules.  
Moreover, if $r\geq r_0$, then 
$\ind_{H_r}^G(\psi_r)$ is a direct summand of $\ind_{U}^G(\psi)$. We record this in the following:

\begin{proposition} \label{P:summand}  For every $r\geq r_0$, $\ind_{H_r}^G(\psi_r)$ is a direct $G$-invariant summand of  $\ind_{U}^G(\psi)$:  
\[ 
\ind_{U}^G(\psi) \cong  \ind_{H_r}^G(\psi_r)  \oplus  \ind_{H_r}^G(\psi_r)_c.   
\] 
\end{proposition} 

\begin{proposition}   \label{P:vanishing} Fix $r\geq r_0$. For almost all $s \geq r$, $(\ind_{H_s}^G(\psi_s)_c)^{K_r}$ is trivial. 
\end{proposition} 
\begin{proof}  The key is the following lemma: 
 \begin{lemma}  
 Let $r\geq r_0$. 
 Let $\pi$ be an irreducible Whittaker generic $G$-module such that $\pi^{K_r}\neq 0$. There exists a positive integer $m$, independent of 
 $\pi$, such that $\pi^{H_{mr}\psi_{mr}} \neq 0$. 
 \end{lemma} 
 \begin{proof} 
 The first step in the proof is a reduction to supercuspidal representations. Let $P=MN$ be a standard parabolic subgroup of block upper-triangular matrices. 
 Assume that $\pi$ is a Whittaker generic 
 subquotient of ${\mathrm{ Ind}}_{P^{\top}}^G\sigma$, where $P^{\top}$ is the transpose of $P$. 
  Let $K=\GL_n(R)$. Using $G=P^{\top} K$ and normality of $K_r$ in $K$, it is easy to see that $\pi^{K_r}\neq 0$ implies 
 that $\sigma^{K^M_r}\neq 0$ where $K^M_r=K_r\cap M$. Now assume that $\sigma^{H_s^M, \psi_s^M}\neq 0$ where $H^M_s=H_s\cap M$ and 
 $\psi_s^M$ is the restriction of $\psi_s$ to $H_s^M$. 
 Let $v\in \sigma^{H_s^M, \psi_s^M}$, and define $f\in {\mathrm {Ind}}_{P^{\top}}^G\sigma$, supported on $P^{\top} (H_s\cap N)$, such that $f(1)=v$ and that it is 
 right $(\psi_s)|_{H_s\cap N}$-invariant. Then $f\in ({\mathrm {Ind}}_{P^{\top}}^G\sigma)^{H_s,\psi_s}$. This type must belong to the Whittaker generic subquotient of the induced representation by injectivity of the map $i_s$. 
 
 It remains to deal with  supercuspidal $\pi$.  Let $\ell$ be a Whittaker functional on $\pi$, and for every $v\in \pi$ we have a Whittaker function 
 $f_v(g)= \ell(\pi(g)v)$. Let $T(r) \subset T$ be the subset of $t=(t_1, \ldots, t_n)$ such $1/q^{(2m-2)r} \leq | t_i/t_{i+1} | \leq q^{(2m-2)r}$, for all $i$.  
 By Theorem 2.1 in \cite{La},  there exists 
 $m$, independent of $\pi$, such that $f_v$ is supported on $UT(r) K$ for all $v\in \pi^{K_r}$. Since $f_v$ is non-zero, for a non-zero $v$, 
 there exists $t\in T(r)$ and $k\in K$ such that $\ell(\pi(tk) v) \neq 0$.   Since $K$ normalizes $K_r$, $\pi(k) v\in \pi^{K_r}$.  It follows that 
 $\pi(tk)v$ is fixed by $tK_r t^{-1}$. Observe that this group contains $H_{mr} \cap B^{\top}$, by the definition of $T(r)$, hence 
 \[ 
 w= \mathrm{vol}( H_{mr} \cap U)^{-1} \int_{H_{mr} \cap U} \bar \psi_U(u) \pi(u) \pi(tk) v \in \pi^{H_{mr},\psi_{mr}} 
 \] 
 and it is non-zero since $\ell(w)= \ell(\pi(tk)v) \neq 0$. The lemma is proved. 
  \end{proof} 
   Take $s \geq mr$, where $m$ is as in the lemma. Recall that, by  \cite{BH}, Bernstein's components of $\ind_{U}^G \psi_U$ are finitely generated and hence admit 
   irreducible quotients. Thus, if $(\ind_{H_s}^G(\psi_s)_c)^{K_r} \neq 0$ then 
$\ind_{H_s}^G(\psi_s)_c$ has an irreducible quotient $\pi$ such that $\pi^{K_r} \neq 0$. 
 Then $\pi^{H_s,\psi_s} \neq 0$ by the lemma, and 
 hence $\dim_G(\ind_{U}^G \psi_U, \pi) \geq 2$, by Proposition  \ref{P:summand},  a contradiction. 
 \end{proof}

\begin{proposition} \label{C:rodier} 
For every $G$-module $\pi$ generated by $\pi^{K_r}$,  and every vector space $\sigma$
\[ 
\Hom_G(\sigma\otimes \ind_U^G \psi_U, \pi) \cong \Hom(\sigma, \pi_{U,\psi_U}).  
\] 
\end{proposition} 
\begin{proof}  By  Propositions   \ref{P:summand}  and  \ref{P:vanishing}, for almost all $s\geq r $, 
\[ 
\Hom_G (\sigma \otimes \ind_U^G \psi_U, \pi)\cong 
\Hom_G(\sigma\otimes \ind_{H_s}^G(\psi_s), \pi). 
\] 
Let $\mathbb C[G]$ denote the group algebra of $G$. Then we can write
\[ 
\sigma\otimes \ind_{H_s}^G(\psi_s)\cong \ind_{H_s}^G(\sigma\otimes \psi_s) \cong \mathbb C[G]\otimes_{\mathbb C[H_s]} (\sigma \otimes \psi_s). 
\] 
Hence, by the Frobenius reciprocity, 
\[ 
\Hom_G(\sigma\otimes \ind_{H_s}^G(\psi_s), \pi) \cong \Hom(\sigma, \pi^{H_s, \psi_s}). 
\]  
Now observe that the starting space $\Hom_G (\sigma \otimes \ind_U^G \psi_U, \pi)$ does not depend on $s$. It follows that the spaces $\pi^{H_s, \psi_s}$ are 
isomorphic for almost all $s$. In particular, $\pi^{H_s, \psi_s}\cong \pi_{U,\psi_U}$ for such $s$.  Hence 
\[ 
\Hom_G (\sigma \otimes \ind_U^G \psi_U, \pi)
\cong \Hom(\sigma,(\pi)_{U,\psi_U}). 
\] 
\end{proof}

\subsection{Second adjointness} 
Now we are ready to prove Lemma \ref{lem second adjointness bz der}. We resume using the notation from the main 
body of the paper, in particular, $G_n=\GL_n(F)$, $U_n$ is the group of upper-triangular unipotent matrices, and $P_{n-i}=M_{n-i} N_{n-i}$ the standard maximal 
parabolic subgroup of block upper-triangular matrices with the Levi $G_{n-i}\times G_i$. Let $\pi$ be a smooth representation of $G_n$ generated by vectors fixed by the 
$r$-th principal congruence subgroup in $G_n$, and 
$\sigma$ a smooth representation of $G_{n-i}$, as in the statement of the lemma. Using the induction in stages, 
\[ 
\mathrm{ind}_{R_{n-i}}^{G_n}(\sigma \otimes \bar\psi_{i}) \cong \mathrm{Ind}_{P_{n-i}}^{G_{n}}(\sigma \boxtimes \ind_{U_i}^{G_i}(\bar \psi_i)).
\] 
By the second adjointness isomorphism for parabolic induction, due to Bernstein, 
\[ 
\Hom_{G_n}(\mathrm{Ind}_{P_{n-i}}^{G_{n}}(\sigma \boxtimes \ind_{U_i}^{G_i}(\bar \psi_i)),\pi )\cong \Hom_{G_{n-i}\times G_i} (\sigma \boxtimes \ind_{U_i}^{G_i}(\bar \psi_i), \pi_{N_{n-i}^{\top}}). 
\] 
It is easy to see that $\pi_{N_{n-i}^{\top}}$, as a $G_i$-module, is also generated by vectors fixed by the $r$-th principal congruence subgroup in $G_i$. Thus we can apply 
Proposition \ref{C:rodier} to $G_i$ to derive 
\[ 
\Hom_{G_{n-i}\times G_i} (\sigma \boxtimes \ind_{U_i}^{G_i}(\bar \psi_i), \pi_{N_{n-i}^{\top}}) \cong \Hom_{G_{n-i}} (\sigma, {}^{(i)}\pi). 
\]

\section{Acknowledgment} 

A part of this collaboration was carried out at the Weizmann Institute of Science during a program on the representation theory of reductive groups in 2017. 
Both authors would like to thank the organizers, Avraham Aizenbud, Joseph Bernstein, Dmitry Gourevitch and Erez Lapid, for their kind invitation to participate in the program. 
The second author is partially supported by an NSF grant DMS-1901745.

\end{document}